\documentclass{aims-ppn}
\usepackage{amsmath}
\usepackage{paralist}
\usepackage{graphics}
\usepackage{epsfig}
\usepackage{graphicx} \usepackage{epstopdf}
\usepackage[colorlinks=true]{hyperref}
\usepackage{tikz}
\usepackage{mathrsfs}
\usepackage{setspace}

1

\newtheorem{theorem}{Theorem}

\newtheorem{lemma}[theorem]{Lemma}
\newtheorem{proposition}[theorem]{Proposition}

\theoremstyle{definition}

\newcommand{\msc}[1]{\href{https://mathscinet.ams.org/mathscinet/search/mscbrowse.html?sk=default&sk=#1&submit=Chercher}{#1}}
\newcommand{\be}[1]{\begin{equation}\label{#1}}
\newcommand{\ee}{\end{equation}}
\renewcommand{\(}{\left(}
\renewcommand{\)}{\right)}
\newcommand{\R}{{\mathbb R}}
\newcommand{\N}{{\mathbb N}}

\newcommand{\ird}[1]{\int_{\R^d}{#1}\,dx}
\newcommand{\nrm}[2]{\|{#1}\|_{\mathrm L^{#2}(\S^d)}}

\renewcommand{\S}{\mathbb{S}}
\newcommand{\iS}[1]{\int_{\S^d}{#1}\,d\mu}
\newcommand{\iSn}[1]{\int_{\S^d}{#1}\,d\mu_n}

\renewcommand{\L}{{\mathcal L}_n\,}

\newcommand{\LL}{{\mathcal L}_{\varepsilon,n}\,}

\newcommand{\izz}[1]{\int_{-1}^1{#1}\,d\nu_{\varepsilon,n}}
\newcommand{\scal}[2]{\left\langle{#1},{#2}\right\rangle}
\newcommand{\ix}[1]{\int_{-1}^1{#1}\,d\nu_{\kern -0.5pt n}}
\newcommand{\nrmx}[2]{\|#1\|_{#2}}

\newcommand{\irdsph}[3]{\int_0^\infty\kern-5pt\int_{\S^{d-1}}#2\,{#1}^{\kern1pt #3-1}\,d{#1}\,d\omega}
\newcommand{\D}[1]{\nabla_x\kern1pt#1}
\newcommand{\DD}[1]{\mathsf D\kern1pt#1}

\newcommand{\DDD}[1]{\mathsf D_x\kern1pt#1}

\newcommand{\irdsphDEL}[2]{\int_0^\infty\int_{\S^{d-1}}#1\,r^{\kern1pt #2}\,\frac{dr}{r}\;d\omega}

\definecolor{darkgreen}{rgb}{0,0.4,0}

\title[Ultraspherical interpolation inequalities]{Parabolic methods for ultraspherical interpolation inequalities}
\author[J.~Dolbeault, A.~Zhang]{Jean Dolbeault, and An Zhang}

\email{dolbeaul@ceremade.dauphine.fr}
\email{anzhang@pku.edu.cn}

\thanks{$^*$ Corresponding author: A.~Zhang}

\begin{document}

\maketitle

\thispagestyle{empty}

\centerline{\scshape Jean Dolbeault}
\smallskip{\footnotesize\centerline{Ceremade, UMR CNRS n$^\circ$~7534}
\centerline{Universit\'e Paris-Dau\-phine, PSL Research University}
\centerline{Place de Lattre de Tassigny, 75775 Paris Cedex~16, France}}
\medskip

\centerline{\scshape An Zhang}
\smallskip{\footnotesize\centerline{School of Mathematical Sciences}
\centerline{Beihang University}
\centerline{No.~ 37 Xueyuan Road, Haidian, Beijing, China, 100191, China}}
\bigskip

\begin{abstract} The \emph{carr\'e du champ} method is a powerful technique for proving interpolation inequalities with explicit constants in presence of a non-trivial metric on a manifold. The method applies to some classical Gagliardo-Nirenberg-Sobolev inequalities on the sphere, with optimal constants. Very nonlinear regimes close to the critical Sobolev exponent can be covered using nonlinear parabolic flows of porous medium or fast diffusion type. Considering power law weights is a natural question in relation with symmetry breaking issues for Caffarelli-Kohn-Nirenberg inequalities, but regularity estimates for a complete justification of the computation are missing. We provide the first example of a complete parabolic proof based on a nonlinear flow by regularizing the singularity induced by the weight. Our result is established in the simplified framework of a diffusion built on the ultraspherical operator, which amounts to reduce the problem to functions on the sphere with simple symmetry properties.\end{abstract}
\bigskip

\begin{center}\parbox{10.2cm}{

\medskip


}\end{center}\bigskip



\section{Introduction and main results}\label{Sec:Intro}

\noindent The \emph{Gagli\-ardo-Nirenberg-Sobolev interpolation inequality},
\[
\tfrac{p-2}d\,\nrm{\nabla u}2^2+\nrm u2^2\ge\nrm up^2\quad\forall\,u\in\mathrm H^1(\S^d,d\mu)\,,
\]
is usually attributed to W.~Beckner~\cite{MR1230930} but can also be found in~\cite[Corollary~6.1]{BV-V}. With some restriction on $p$, it was even proved earlier in~\cite{MR772092,Bakry-Emery85,MR808640} by the \emph{carr\'e du champ} method. Here $d\mu$ is the uniform probability measure on the $d$-dimensional sphere~$\S^d$ corresponding to the measure induced by Lebesgue's measure on $\S^d\subset\R^{d+1}$, up to a normalization constant $1/|\S^d|$, and $p$ is an arbitrary exponent in the range $2<p\le2^*$ if $d\ge3$, where the critical exponent is defined as $2^*:=2\,d/(d-2)$. The inequality holds true for any finite $p>2$ if $d=1$ or $2$: in this case, we adopt the convention that $2^*=+\infty$ and shall write that $p\in(2,2^*)$. The norms are defined with respect to~$d\mu$, that is,
\[
\nrm uq:=\(\iS{|u|^q}\)^{1/q}\,.
\]
We can rewrite the interpolation inequality as
\be{Ineq:GNS}
\nrm{\nabla u}2^2\ge\frac d{p-2}\(\nrm up^2-\nrm u2^2\)\quad\forall\,u\in\mathrm H^1(\S^d,d\mu)\,,
\ee
which now makes sense for any $p\in[1,2)\cup(2,2^*]$ if $d\ge 3$, and for any $p\in[1,2)\cup(2,+\infty)$ if $d=1$ or $d=2$. The case $p=2^*$ is the \emph{critical} case of the inequality and the usual Sobolev inequality on $\R^d$, with optimal constant, is recovered using the stereographic projection. The case $p=2$ has of course to be excluded, but it is easy to pass to the limit in~\eqref{Ineq:GNS} as $p\to2$ and obtain the logarithmic Sobolev inequality on~$\S^d$, \emph{i.e.},
\be{Ineq:logSob}
\nrm{\nabla u}2^2\ge\frac d2\iS{|u|^2\,\log\(\frac{|u|^2}{\nrm u2^2}\)}\quad\forall\,u\in\mathrm H^1(\S^d,d\mu)\setminus\{0\}\,.
\ee
For brevity, we shall refer to this case as the $p=2$ case. Equality in~\eqref{Ineq:GNS} and~\eqref{Ineq:logSob} is achieved by constant functions. The constants $d/(p-2)$ in~\eqref{Ineq:GNS} and $d/2$ in~\eqref{Ineq:logSob} are optimal. If $\varphi$ is an eigenfunction (a spherical harmonic function) of the Laplace-Beltrami operator $\Delta$ on~$\S^d$ such that $-\,\Delta\varphi=d\,\varphi$, we use $u=1+\varepsilon\,\varphi$ as a test function. A Taylor expansion at order two in $\varepsilon\to0$ shows that~\eqref{Ineq:GNS} and~\eqref{Ineq:logSob} cannot hold with larger constants. This is of course the easy part of the result. The really difficult part is to prove that~\eqref{Ineq:GNS} and~\eqref{Ineq:logSob} are true respectively with the constants $d/(p-2)$ and $d/2$. Over the years various proofs have been proposed:
\begin{enumerate}
\item[1)] The strategy of W.~Beckner in~\cite{MR1230930} relies on the sharp Hardy-Littlewood-Sobolev inequalities of E.H.~Lieb in~\cite{MR717827}, Legendre duality and monotonicity properties of the coefficients in a decomposition on spherical harmonics.
\item[2)] The proof of M.-F.~Bidaut-V\'eron and L.~V\'eron in~\cite{BV-V} is based on nonlinear elliptic methods initiated for the study of critical points in~\cite{MR615628}.
\item[3)] The \emph{carr\'e du champ} method introduced by D.~Bakry and M.~\'Emery in~\cite{Bakry-Emery85} uses semi-groups defined by linear parabolic flows and it is restricted for $d\ge2$ to $p\le2^\#$ where $2^\#:=\frac{2\,d^2+1}{(d-1)^2}<2^*$. We refer to~\cite{MR3155209} for a general presentation of the method and results which go far beyond functional inequalities, and to~\cite{Gentil_2021,dupaigne2021conformal} for recent developments.
\item[4)] There are various other proofs based on stereographic projection and symmetrizations, transformations of Emden-Fowler type as in~\cite{MR1132767} or~\cite{Dolbeault06082014}, mass transport in~\cite{MR2032031}, inversion symmetry and direct use of spectral properties in~\cite{MR2659680}, at least for some specific values of $p$, which we will not discuss here.
\end{enumerate}
It turns out that the approaches 2) and 3) are based on similar computations as was discussed by J.~Demange in~\cite{MR2381156} and the restriction $p\le2^\#$ can be lifted by extending the \emph{carr\'e du champ} method to nonlinear parabolic equations of porous medium and fast diffusion type. The overall strategy is based on the \emph{entropy methods}: see~\cite{Carrillo2000,Carrillo2001,Carrillo2003} for early results on nonlinear parabolic equations by the \emph{carr\'e du champ} method,~\cite{Dolbeault:2021wb} for a recent review and~\cite{BDNS} for further details. The method has been made systematic in a series of papers~\cite{DEKL2012,1302,DEKL2014,1504, Dolbeault_2020b} with important consequences on \emph{symmetry and symmetry breaking} issues studied in~\cite{DEL-2015,Dolbeault2017133,1703}, on which we will come back below.

The point of view of 3) is beautiful as it provides a clear structure and, at least from a formal point of view, a gradient flow interpretation. This allows to order computations which are otherwise extremely complicated in 2) and involve a rather non-intuitive use of test functions. However, there is a price to pay as the regularity needed to justify all necessary integrations by parts is not \emph{a priori} granted for the solutions of the nonlinear parabolic equations when singular weights are present. A major advantage of 2) is indeed that the elliptic regularity theory can be invoked to justify the computations.

In this paper, we are interested in \emph{weighted interpolation inequalities} of the form
\be{Ineq:GNSw}
\iSn{|\nabla u|^2}\ge\frac{\mathscr C_{n,p}}{p-2}\(\(\iSn{|u|^p}\)^\frac2p-\iSn{|u|^2}\)\quad\forall\,u\in\mathrm H^1(\S^d,d\mu_n)
\ee
where $d\mu_n=Z_n^{-1}\,\rho^{n-d}\,d\mu$ is a probability measure built on $d\mu$ with a weight $\rho^{n-d}$ where $\rho=\rho(x)$ is such that $\rho(x)^2=1-(x\cdot\mathsf e)^2$, for some given, fixed $\mathsf e\in\S^d$, and $n$ is generically \emph{not an integer}. A motivation to study of~\eqref{Ineq:GNSw} is for instance to consider the case $0<n<d$ so that the weight has a singularity at $x=\mathsf e$, or the case $n>d$ so that the operator associated to the Dirichlet form degenerates at $x=\mathsf e$. Because of the rotation invariance, there is no loss of generality in assuming that $\mathsf e$ is the north pole and we shall consider cylindrical coordinates $(z,\omega)\in[-1,1]\times\S^{d-1}$ such that $x=z\,\mathsf e+\rho\,\omega$ where $\rho=\sqrt{1-z^2}$. A deep reason for studying~\eqref{Ineq:GNSw} comes from its relation with some \emph{Caffarelli-Kohn-Nirenberg inequalities} (see~\cite{dupaigne2021conformal} for more details), which goes as follows. Let us consider the stereographic projection from $\S^d$ (with cylindrical coordinates) onto $\R^d$ (with spherical coordinates) given by
\[
[-1,1]\times\S^{d-1}\ni(z,\omega)\mapsto(s,\omega)\in[0,+\infty)\times\S^{d-1}
\]
where
\[\textstyle
s=\sqrt{\frac{1+z}{1-z}}\;\Longleftrightarrow\;z=\frac {s^2-1}{s^2+1}=1-\frac2{s^2+1}\quad\mbox{and}\quad\rho=\sqrt{1-z^2}=\frac{2\,s}{s^2+1}\,.
\]
To the function $u\in\mathrm H^1(\S^d,d\mu_n)$, we associate the function on $w$ on $\R^d$ such that
\[
w(s,\omega)=\big(\tfrac2{1+s^2}\big)^\frac{n-2}2\,u(z,\omega)\,.
\]
Here we consider the case $p=2\,n/(n-2)$. Inequality~\eqref{Ineq:GNSw} becomes
\[
\int_{\R^d}|\nabla w|^2\,s^{n-1}\,ds\,d\omega\ge\frac{\mathscr C_{n,p}}{p-2}\,|\S^d|^{1-\frac2p}\(\int_{\R^+\times\S^{d-1}}|w|^p\,s^{n-1}\,ds\,d\omega\)^\frac2p
\]
which can be rewritten simply as
\be{CKN2}
\ird{|\nabla w|^2\,|x|^{n-d}}\ge\frac{\mathscr C_{n,p}}{p-2}\,|\S^d|^{1-\frac2p}\(\ird{|w|^p\,|x|^{n-d}}\)^\frac2p\,.
\ee
To relate~\eqref{CKN2} with the Caffarelli-Kohn-Nirenberg inequalities, let us do one more change of variables as in~\cite{DEL-2015} and define
\[
f(r,\omega)=w(s,\omega)\quad\mbox{with}\quad s=r^\alpha\,,\quad\alpha=\frac{d-2-2\,\mathsf a}{n-2}=\frac14\,(d-2-2\,\mathsf a)\,(p-2)
\]
for some new parameter $\mathsf a\in(0,(d-2)/2)$. With
\[\textstyle
\mathsf b:=\mathsf a+1-\frac dn\,,\quad\mathcal C_{\mathsf a,\mathsf b}=\frac{\mathscr C_{n,p}}{p-2}\,\frac{|\S^d|^{1-\frac2p}}{\alpha^{1-\frac2p}}\quad\mbox{and}\quad\mathsf D_\alpha f:=\(\tfrac1\alpha\,\tfrac{\partial f}{\partial r},\tfrac1r\,\nabla_\omega f\),
\]
we can reformulate~\eqref{Ineq:GNSw} as
\be{CKN}
\ird{\frac{|\mathsf D_\alpha f|^2}{|x|^{2\,\mathsf a}}}\ge\mathcal C_{\mathsf a,\mathsf b}\(\ird{\frac{|f|^p}{|x|^{\mathsf b\,p}}}\)^\frac2p\,.
\ee
These equivalences are of course formal: the range of the exponents has limitations and in the case $n>d\ge3$ for instance, we should consider only $p=2\,p/(d-2+2\,\mathsf b-2\,\mathsf a)$ with the restriction $0\le\mathsf a\le\mathsf b\le\mathsf a+1<d/2$. Moreover, the functional setting has to be properly defined by considering either functions $f$ which are smooth with compact support in $\R^d\setminus\{0\}$, or a space obtained by completion with respect to the norm corresponding to the left-hand side in~\eqref{CKN}. A standard issue is to decide whether optimality is achieved among radial functions or not. This can be done by elliptic methods: see~\cite{Catrina2001,Felli2003,DEL-2015}. However, the intuition of the proof of~\cite{DEL-2015} relies on the \emph{carr\'e du champ} method and, as it is discussed at length in~\cite{DEL-JEPE}, a complete parabolic proof is so far missing: \emph{the extension of the carr\'e du champ method to weighted, nonlinear diffusion equations} is only formal.

In this paper, we focus on a parabolic approach of~\eqref{Ineq:GNSw} based on a \emph{regularization of the weight} and the \emph{carr\'e du champ method}. The discussion of the symmetry issues will be ignored as it involves various cases and requires elementary although very technical computations. For this reason, we shall restrict our study to symmetric functions on~$\S^d$ depending only on the $z$ variable. In that setting, the usual Laplace-Beltrami operator becomes the \emph{$d$-ultraspherical operator} while the operator associated with~\eqref{Ineq:GNSw} is now an \emph{$n$-ultraspherical operator} for some real valued parameter $n$.

Let us consider functions of the variable $z\in[-1,1]$. We define the probability measure
\[
d\nu_{\kern -0.5pt n}(z):=Z_n^{-1}\,\rho^{n-2}\,dz\quad\mbox{where}\quad Z_n:=\frac{\sqrt\pi\,\Gamma(\frac n2)}{\Gamma(\frac{n+1}2)}\,,\quad\rho=\sqrt{1-z^2}\,,
\]
for some real parameter $n>0$, the space $\mathrm L^2([-1,1],d\nu_{\kern -0.5pt n}))$ with scalar product $\scal{f_1}{f_2}:=\ix{f_1\,f_2}$ and the corresponding norm $\nrmx f2:=\scal ff^{1/2}$. For any $q>2$, let
\[
\nrmx fq=\(\ix{|f|^q}\)^{1/q}\,.
\]
The self-adjoint \emph{ultraspherical} operator
\[
\L f:=(1-z^2)\,f''-n\,z\,f'=\rho^2\,f''-n\,z\,f'
\]
has the fundamental property
\be{Fundamental}
\scal{f_1}{\L f_2}=-\ix{f_1'\,f_2'\,\rho^2}
\ee
and if $f=f_1=f_2$, we have $-\scal f{\L f}=\nrmx{\rho\,f'}2^2$. We denote by $\mathrm H^1([-1,1],d\nu_{\kern -0.5pt n})$ the space of the functions $f\in\mathrm L^2([-1,1],d\nu_{\kern -0.5pt n})$ such that $\nrmx{\rho\,f'}2^2+\nrmx f2^2<\infty$. We extend the notion of \emph{critical exponent} using
\be{Critical:n}
2^*:=\frac{2\,n}{n-2}\quad\forall\,n>2
\ee
and adopt the convention that $2^*=+\infty$ if $n\le2$.
\begin{theorem}\label{Thm:Ultra} Let $n>0$. Then for any $p\in[1,2)\cup(2,2^*)$ if $n\le2$ and any $p\in[1,2)\cup(2,2^*]$ if $n>2$, we have
\be{InterpUS}
\nrmx{\rho\,f'}2^2\ge\frac n{p-2}\(\nrmx fp^2-\nrmx f2^2\)\quad\forall\,f\in\mathrm H^1([-1,1],d\nu_{\kern -0.5pt n})\,,
\ee
while for $p=2$, we have
\be{nLogSob}
\nrmx{\rho\,f'}2^2\ge\frac n2\ix{f^2\log\(\frac{f^2}{\nrmx f2^2}\)}\quad\forall\,f\in\mathrm H^1([-1,1],d\nu_{\kern -0.5pt n})\,.
\ee
The constants in~\eqref{InterpUS} and~\eqref{nLogSob} are sharp.\end{theorem}
The logarithmic Sobolev inequality~\eqref{nLogSob} has been proved in~\cite[Theorem~1]{MR674060} for any real number $n\in(0,+\infty)$. It is obtained by taking the limit as $p\to2$ and we assume from now on that $p\neq2$ without further notice. Inequality~\eqref{InterpUS} with $p=2^*$ and a real parameter $n>2$ can be found in~\cite{MR1132767} and in the whole range of $p>2$ in~\cite{MR1231419} with a proof based on elliptic estimates which is very similar to the proof of~\cite{BV-V}. Inspired by~\cite{MR2381156}, results for the ultraspherical operator with $n\in\N$ were proved in~\cite{DEKL2012,DEKL2014,1504} using entropy methods, linear and nonlinear diffusion equations and appropriate versions of the carr\'e du champ method. As far as we know, when $n$ is not an integer, no rigorous proof based on flows has been provided if $n>1$ and $2^\#<p<2^*$, where
\be{Sharp:n}
2^\#:=\frac{2\,n^2+1}{(n-1)^2}
\ee
is the \emph{Bakry-\'Emery exponent}. The ultraspherical operator is considered in~\cite{Bakry-Emery85} and the limitation to the range $p\le2^\#$ for using the heat flow is made clear in~\cite{MR808640}. This limitation can be lifted by considering a nonlinear diffusion equation (see~\cite{1504} for detailed explanations and earlier references therein) when~$n$ is an integer (we refer to~\cite{MR2282669,MR2286292} for the theory of fast diffusion and porous medium equations). With $n\in\R^+\setminus\N$, regularity properties and an approximation method are needed to prove Theorem~\ref{Thm:Ultra}.

\section{Ultraspherical operators and \emph{carr\'e du champ} methods}\label{Sec:Ultraspherical}

This section is devoted to the proof of Theorem~\ref{Thm:Ultra}. We start by a few observations which motivate our method.

\subsection{Preliminary remarks}\label{Sec:Prelim}

The case $p=1$ of~\eqref{InterpUS} is the \emph{Poincar\'e or spectral gap inequality} associated with $\L$ and can be written as
\[
-\scal f{\L f}\ge n\,\nrmx f2^2\quad\forall\,f\in\mathrm H^1([-1,1],d\nu_{\kern -0.5pt n})\;\mbox{such that}\;\ix f=0\,.
\]
The ultraspherical operator has purely discrete spectrum $\lambda_k=k\,(k+n-1)$, $k\in\N$, with eigenfunctions $f_0=1$ and $f_1(z)=z$ associated with $k=0$ and $k=1$. A basis of eigenfunctions is provided by the \emph{Gegenbauer} or \emph{ultraspherical polynomials}. The optimality of the constants in~\eqref{InterpUS} and~\eqref{nLogSob} is easy to check using $f=f_0+\varepsilon\,f_1$ as a test function and then considering the limit as $\varepsilon\to0_+$. See~\cite{DEKL2014,Dolbeault_2020b,Frank-2021} for the next order term.

Equality in~\eqref{InterpUS} and~\eqref{nLogSob} is achieved only by constant functions if $p<2^*$ and by $f_{a,b}(z)=a\,|1-b\,z|^{(n-2)/2}$ for some real numbers $a$ and $b$ with~$|b|<1$ if $n>2$ and $p=2^*$. The fact that the optimizers $f_{a,b}$ are not limited to constant functions in the critical case reflects the invariance under conformal transformations. It is elementary to check that for an appropriate choice of $t\mapsto(a(t),b(t))$, $f_{a(t),b(t)}$ solves the fast diffusion equation that will be considered in Section~\ref{Sec:NLUltraspherical}. The family $(f_{a,b})$ can be characterized by considering the remainder term in the \emph{carr\'e du champ} method. See~\cite[Proposition~4.1]{1504} when $n\in\N$.

As already mentioned, when $n$ is an integer, $\L$ can be interpreted as the restriction of the Laplace-Beltrami operator on~$\S^n$ to symmetric functions depending only on $z=x\cdot\mathsf e$, $x\in\S^n$, where $\mathsf e\in\S^n$ is a given, fixed unit vector in $\R^{n+1}$. An important consequence is that the \emph{Comparison Principle} applies and the whole regularity theory of classical diffusion equations also applies. Arguing by density, one can indeed reduce the proof of the inequalities to functions which are bounded from above and from below by positive constants. Solutions of the evolution equations are then smooth and bounded away from $0$ for any positive time and integrations by parts can be performed with no special precaution. However, when $n\not\in\N$, some additional care is needed. We shall distinguish two cases:
\begin{enumerate}
\item If we use linear flows which preserve finite dimensional spaces generated by Gengenbauer polynomials, then we can argue by density. Again no special precautions are needed. All values of $p\ge1$ are covered if $n\le1$, while this induces the limitation $1\le p\le2^\#$, with $2^\#$ defined by~\eqref{Sharp:n} if $d>1$. Details are given in Section~\ref{Sec:heatUltraspherical}.
\item If either $1<n\le2$ and $2^\#<p<+\infty$ or $n>2$ and $2^\#<p\le2^*$ with $2^*$ defined by~\eqref{Critical:n}, we have to use a nonlinear diffusion equation: see~\cite[Proposition~4.2]{1504} when $n\in\N$. Based on a formal computation in Section~\ref{Sec:NLUltraspherical} and a regularization in Section~\ref{Sec:Regularization}, Theorem~\ref{Thm:Ultra} is proved in Section~\ref{Sec:proofUltra}.
\end{enumerate}
In Case 1, it is possible to assume as in~\cite[page~71]{MR3155209} that ``the carr\'e du champ operator~$\Gamma$ is defined on a suitable algebra of functions in the $\mathrm L^2(\S^d,\,d\mu_n)$-domain'' of the diffusion operator. When the weighted heat flow has to be replaced by its nonlinear counterpart, which occurs in Case 2, this is not possible anymore.

\subsection{A computation based on the heat flow}\label{Sec:heatUltraspherical}

We consider the case $p\in(2,+\infty)$ if $n\le1$ and $p\le2^\#$ otherwise. The two following identities are classical but rely on computations that will be generalized in Section~\ref{Sec:NLUltraspherical}, so let us give some details.
\begin{lemma}\label{Lem:TwoId} Let $n>0$ be a real number. For any positive function $u\in C^2$ such that $u'(\pm1)=0$, we have
\begin{align}
&\label{Gamma2}\ix{(\L u)^2}=\ix{|u''|^2\,\rho^4}+n\ix{\rho^2\,|u'|^2}\,,\\
&\label{L-Gamma}\scal{\frac{|u'|^2}u\,\rho^2}{\L u}=\frac n{n+2}\ix{\frac{|u'|^4}{u^2}\,\rho^4}-\,2\,\frac{n-1}{n+2}\ix{\frac{|u'|^2\,u''}u\,\rho^4}\,.
\end{align}
\end{lemma}
\begin{proof} Using the commutator $\left[\tfrac d{dz},\L\right]u=-\,2\,z\,u''-\,n\,u'$, we observe that
\begin{align*}
\ix{(\L u)^2}&=-\ix{u'\,(\L u)'\,\rho^2}\\
&=-\ix{u'\,(\L u')\,\rho^2}-\ix{u'\(\left[\tfrac d{dz},\L\right]u\)\rho^2}\\
&=\ix{u''\,\big(u'\,\rho^2\big)'\,\rho^2}+\ix{u'\(2\,z\,u''+\,n\,u'\)\rho^2}\\
&=\ix{|u''|^2\,\rho^4}+n\ix{\rho^2\,|u'|^2}\,,
\end{align*}
which establishes~\eqref{Gamma2}. Identity~\eqref{L-Gamma} is obtained by writing that
\begin{multline*}
\ix{(\L u)\,\frac{|u'|^2}u\,\rho^2}=-\ix{u'\(\frac{|u'|^2}u\,\rho^2\)'\,\rho^2}\\
=\ix{\frac{|u'|^4}{u^2}\,\rho^4}-\,2\ix{u''\,\frac{|u'|^2}u\,\rho^4}+2\ix{\frac{{u'}^3}u\,\rho^2\,z}
\end{multline*}
and using
\begin{multline*}
2\ix{\frac{{u'}^3}u\,\rho^2\,z}=-\,\frac2{n+2}\int_{-1}^{+1}\(\rho^{n+2}\)'\,\frac{{u'}^3}u\,\frac{dz}{Z_n}=\frac2{n+2}\ix{\(\frac{{u'}^3}u\)'\,\rho^4}\\
=\frac6{n+2}\ix{u''\,\frac{|u'|^2}u\,\rho^4}-\,\frac2{n+2}\ix{\frac{|u'|^4}{u^2}\,\rho^4}\,.
\end{multline*}
In both cases, we repeatedly use~\eqref{Fundamental} and no boundary term appears because of the assumption \hbox{$u'(\pm1)=0$}.
\end{proof}

\medskip\begin{proof}[Proof of Theorem~\ref{Thm:Ultra} if $p\in[1,2)\cup(2,+\infty)$, $n=1$, or \texorpdfstring{$p\in[1,2)\cup(2,2^{\#}]$}{}, $n\neq1$]~\\
Let us start by a few formal computations and consider the flow
\be{Eqn:UltraHeat}
\frac{\partial u}{\partial t}=\L u+\kappa\,\frac{|u'|^2}u\,\rho^2
\ee
with initial datum $u(t=0,\cdot)=u_0$ and notice that
\[
\frac d{dt}\ix{u^p}=p\,(\kappa-p+1)\ix{u^{p-2}\,\rho^2\,|u'|^2}
\]
so that $\ix{|u(t,\cdot)|^p}=\ix{|u_0|^p}=:\overline u^p$ for any $t\ge0$ if $\kappa=p-1$. With $v=u^p$, the flow~\eqref{Eqn:UltraHeat} is given by
\be{Flow-v}
\frac{\partial v}{\partial t}=\L v
\ee
and $\overline u^p=\ix v$ is the mass, which does not depend on $t$. The flow~\eqref{Flow-v} is the restriction of the heat flow on $\S^n$ to symmetric functions depending only on $z$, if $n$ is an integer. In the general case, if we assume that $u'(t,\pm1)=0$ for any $t\ge0$, then straightforward computations show that
\[
\frac 12\,\frac d{dt}\ix{\rho^2\,|u'|^2}=-\ix{\kern-2pt (\L u)\,\frac{\partial u}{\partial t}}=-\ix{\kern-2pt (\L u)\(\L u+\kappa\,\frac{|u'|^2}u\,\rho^2\)}\,,
\]
\begin{multline*}
\frac 12\,\frac d{dt}\ix{|u|^2}=\ix{u\,\frac{\partial u}{\partial t}}=\ix{u\(\L u+\kappa\,\frac{|u'|^2}u\,\rho^2\)}\\
=-\,(\kappa-1)\ix{\rho^2\,|u'|^2}\,.
\end{multline*}
By Lemma~\ref{Lem:TwoId} and using $\kappa-1=p-2$, we obtain
\begin{multline*}
\frac 12\,\frac d{dt}\ix{\(\rho^2\,|u'|^2+\frac n{p-2}\,\big(|u|^2-\overline u^2\big)\)}\\
=-\ix{|u''|^2\,\rho^4}+2\,\frac{n-1}{n+2}\,\kappa\ix{u''\,\frac{|u'|^2}u\,\rho^4}-\frac n{n+2}\,\kappa\ix{\frac{|u'|^4}{u^2}\,\rho^4}
\end{multline*}
where $\overline u:=\(\ix{u^p}\)^{1/p}$ is preserved if $\kappa=p-1$. The r.h.s.~is negative for any $p>1$ if $n=1$ and any $p>1$ such that
\[
p\le\frac{2\,n^2+1}{(n-1)^2}=:2^\#\quad\mbox{if}\quad n\neq1\,.
\]
See Appendix~\ref{Appendix} for details. This computation is not new and goes back to~\cite{Bakry-Emery85,MR808640}.

Hence we have learned that
\[
\frac d{dt}\ix{\(\rho^2\,|u'|^2+\frac n{p-2}\,\big(|u|^2-\overline u^2\big)\)}\le0\,.
\]
It is a standard property of the flow~\eqref{Flow-v} that $v$ converges to $\ix v=\overline u^p$ as $t\to+\infty$, from which we infer that, for any $t\ge0$,
\begin{multline*}
\ix{\(\rho^2\,|u_0'|^2+\frac n{p-2}\,\big(|u_0|^2-\overline u^2\big)\)}\\
\ge\ix{\(\rho^2\,|u'(t,\cdot)|^2+\frac n{p-2}\,\big(|u(t,\cdot)|^2-\overline u^2\big)\)}\\
\ge\lim_{s\to+\infty}\ix{\(\rho^2\,|u'(s,\cdot)|^2+\frac n{p-2}\,\big(|u(s,\cdot)|^2-\overline u^2\big)\)}=0\,.
\end{multline*}
This is precisely the proof of~\eqref{InterpUS} written at any time $t\ge0$ for a solution of~\eqref{Eqn:UltraHeat} with an arbitrary initial datum $u_0\in\mathrm H^1([-1,1],d\nu_{\kern -0.5pt n})$, and in particular at time $t=0$ for~$u_0$. When $n$ is an integer, it is clear that there are no boundary terms because of the interpretation of $\mathcal L_n$ as the restriction of the Laplace-Beltrami operator on~$\S^n$ to functions which depend only on $z$. If $n$ is not an integer, the computations can be justified without assuming that $u'(t,\pm1)=0$ for any $t\ge0$ by considering a decomposition of the functions on Gegenbauer polynomials. For any approximation by a finite sum of such polynomials, integrations by parts can be carried out and the result follows by density of the Gegenbauer polynomials. These remarks justify the above computations and complete the proof in the case under consideration.
\end{proof}

\subsection{A formal computation based on a nonlinear diffusion flow}\label{Sec:NLUltraspherical}

To overcome the limitation $p\le2^\#$, as in~\cite{DEKL2014,1504}, let us consider the nonlinear flow
\be{Eqn:UltraFD}
\frac{\partial u}{\partial t}=u^{2-2\beta}\(\L u+\kappa\,\frac{|u'|^2}u\,\rho^2\)
\ee
for some exponent $\beta$ to be chosen and notice that
\[
\frac d{dt}\ix{u^{\beta p}}=\beta\,p\,\big(\kappa-\beta\,(p-2)-1\big)\ix{u^{\beta(p-2)}\,\rho^2\,|u'|^2}=0
\]
with $\kappa=\beta\,(p-2)+1$ so that $\overline u=\(\ix{u^{\beta p}}\)^{1/(\beta p)}$ is preserved. Then $v=u^{\beta p}$ solves
\[
\frac{\partial v}{\partial t}=\frac 1m\,\L v^m\,,
\]
where $\beta$ and $m$ are related by
\be{Id:mbeta}
m=1+\frac2p\(\frac1\beta-1\).
\ee
Assuming that all integrations by parts can be done without adding boundary terms, simple but formal computations shows that
\begin{multline*}
\frac d{dt}\ix{u^{2\beta}}=2\,\beta\!\ix{u^{2\beta-1}\,\frac{\partial u}{\partial t}}\\
=2\,\beta\!\ix{u\(\L u+\kappa\,\frac{|u'|^2}u\,\rho^2\)}=2\,\beta\,(\kappa-1)\ix{\rho^2\,|u'|^2}\,,
\end{multline*}
\begin{multline*}
\frac d{dt}\ix{\rho^2\big|(u^\beta)'\big|^2}=-\,2\ix{\L(u^\beta)\,u^{\beta-1}\,\frac{\partial u}{\partial t}}\\
=-\,2\,\beta^2\ix{\(\L u+(\beta-1)\,\frac{|u'|^2}u \rho^2\)\(\L u+\kappa\,\frac{|u'|^2}u \rho^2\)}\,,
\end{multline*}
and, after collecting the terms,
\begin{multline*}
\frac 12\,\frac d{dt}\ix{\(\rho^2\,\big|(u^\beta)'\big|^2+\frac n{p-2}\,\(u^{2\beta}-\overline u^{2\beta}\)\)}\\
=-\,\beta^2\ix{\(\L u+(\beta-1)\,\frac{|u'|^2}u \rho^2\)\(\L u+\kappa\,\frac{|u'|^2}u \rho^2\)}\\
+\frac{n\,\beta}{p-2}\,(\kappa-1)\ix{\rho^2\,|u'|^2}\,.
\end{multline*}
Assuming that $u'(t,\pm1)=0$ for any $t\ge0$, we deduce from Lemma~\ref{Lem:TwoId} that
\begin{multline*}
-\frac 1{2\,\beta^2}\,\frac d{dt}\ix{\(\rho^2\,\big|(u^\beta)'\big|^2+\frac n{p-2}\,\(u^{2\beta}-\overline u^{2\beta}\)\)}\\
=\ix{|u''|^2\,\rho^4}-2\,\frac{n-1}{n+2}\,(\kappa+\beta-1)\ix{u''\,\frac{|u'|^2}u\,\rho^4}\\
+\left[\kappa\,(\beta-1)+\,\frac n{n+2}\,(\kappa+\beta-1)\right]\ix{\frac{|u'|^4}{u^2}\,\rho^4}
\end{multline*}
because $\frac{n\,\beta}{p-2}(\kappa-1)=n$ so that the coefficient of $\ix{\rho^2\,|u'|^2}$ vanishes.
\begin{lemma}\label{Lem:Monotonicity-m} Let $n>0$ be a real number with $n\neq1$ and take any $p>2^\#$ such that $p\le2^*$ if $n>2$. With the above notations, there is a choice of $\beta$ such that any smooth solution of~\eqref{Eqn:UltraFD} satisfies
\[
\frac d{dt}\ix{\(\rho^2\,\big|(u^\beta)'\big|^2+\frac n{p-2}\,\(u^{2\beta}-\overline u^{2\beta}\)\)}\le0\,.
\]
\end{lemma}
The precise range $\mathcal R(n,p)$ of the \emph{admissible values of $\beta$} is detailed in Appendix~\ref{Appendix} and we refer to~\cite{DEKL2014,1504,Dolbeault_2020b} for further details. The integrations by parts are justified whenever~$n$ is an integer and $u$ can be considered as a symmetric function on $\S^n$. When $n$ is not an integer, the above computations are formal: it has to be justified that no boundary terms appear at $z=\pm1$ when integrating by parts. This is what we are left with for completing the proof of Theorem~\ref{Thm:Ultra}. Indeed, up to this issue, the monotonicity~of
\[
t\mapsto\ix{\(\rho^2\,\big|(u(t,\cdot)^\beta)'\big|^2+\frac n{p-2}\,\(u(t,\cdot)^{2\beta}-\overline u^{2\beta}\)\)}
\]
and the fact that its limit as $t\to+\infty$ is $0$ provides us with a proof of~\eqref{InterpUS}.

\subsection{Regularization}\label{Sec:Regularization}

We consider the probability measure $d\nu_{\varepsilon,n}$ defined by
\[
\nu_{\varepsilon,n}\,dz=d\nu_{\varepsilon,n}:=Z_{\varepsilon,n}^{-1}\,\zeta_\varepsilon\,d\nu_d\quad\mbox{with}\quad\zeta_\varepsilon=\(1+\varepsilon-z^2\)^\frac{n-d}2\,,
\]
with normalization constant $Z_{\varepsilon,n}$. Here $\varepsilon$ is a positive parameter that is intended to be small and
\[
d=\inf\{k\in\N\,:\,k\ge n\}\,.
\]
On the space $\mathrm L^2([-1,1],d\nu_{\varepsilon,n})$ equipped with the scalar product
\[
\scal{f_1}{f_2}_{\varepsilon,n}:=\int_{-1}^1f_1\,f_2\,d\nu_{\varepsilon,n}\,,
\]
we shall denote by $\nrmx f{q,\varepsilon,n}$ the $\mathrm L^q([-1,1],d\nu_{\varepsilon,n})$ norm of $f$ and define a \emph{generalized ultraspherical} operator by $\scal{f_1}{\LL f_2}=-\izz{f_1'\,f_2'\,\rho^2}$. Explicitly we have:
\be{LL}
\LL f:=(1-z^2)\,f''-\,z\(d+(n-d)\,\frac{1-z^2}{1+\varepsilon-z^2}\)\,f'=(1-z^2)\,f''-\,\ell_{\varepsilon,n}\,f'
\ee
using the notation
\[
\ell_{\varepsilon,n}(z):=z\(n-\varepsilon\,\frac{n-d}{1+\varepsilon-z^2}\).
\]
The natural functional space is
\[
\mathrm H^1([-1,1],d\nu_{\varepsilon,n}):=\left\{f\in\mathrm L^2([-1,1],d\nu_{\varepsilon,n})\,:\nrmx{\rho\,f'}{2,\varepsilon,n}<\infty\right\}\,.
\]
We are now interested in the interpolation inequalities
\be{Ineq:Interpolation}
\nrmx{\rho\,u'}{2,\varepsilon,n}^2\ge\frac\lambda{p-2}\(\nrmx u{p,\varepsilon,n}^2-\nrmx u{2,\varepsilon,n}^2\)\quad\forall\,u\in\mathrm H^1([-1,1],d\nu_{\varepsilon,n})
\ee
in the range
\be{Range}
p>2^\#:=\frac{2\,n^2+1}{(n-1)^2}\quad\mbox{if}\quad n\neq1\,,\quad\mbox{and}\quad p\le2^*=\frac{2\,n}{n-2}\quad\mbox{if}\quad n>2\,.
\ee
Our goal is to determine the largest possible value of $\lambda$ such that~\eqref{Ineq:Interpolation} holds in the asymptotic regime as $\varepsilon\to0_+$.

\medskip Let us consider the regularized nonlinear flow
\be{Eqn:UltraFDreg}
\frac{\partial u}{\partial t}=u^{2-2\beta}\(\LL u+\kappa\,\frac{|u'|^2}u\,\rho^2\)
\ee
with $\kappa=\beta\,(p-2)+1$. Our first observation is a regularity result.
\begin{lemma}\label{Lem:Regreg} Assume that $d\ge1$ is an integer, $n\in(0,1)\cup(1,+\infty)$ is such that $d-1<n<d$, $p>2^\#$,
\be{RangeBeta}
1<\beta<\frac n{n-p}\quad\mbox{if}\quad p<n\quad\mbox{and}\quad\beta\neq\frac{n+2}{n+2-p}\,.
\ee
For any $h_0\in(0,1)$ small enough and any $h_1>0$, there exists some $\varepsilon_0>0$ such that, for any solution $u$ of~\eqref{Eqn:UltraFDreg} with initial datum $u_0$, such that $h_0<u_0<1/h_0$ and $|u_0'|\le h_1$ and for any $\varepsilon\in(0,\varepsilon_0)$, then we also have
\[
h_0<u(t,\cdot)<\frac1{h_0}\quad\mbox{and}\quad|u'(t,\cdot)|\le h_1\quad\forall\,t>0\,.
\]
\end{lemma}
\begin{proof} For any $\varepsilon>0$, the operator $\LL$ can be considered as an elliptic operator on~$\S^d$, acting on the special class of symmetric functions which depend only on~$z$. The standard theory of parabolic equations applies and $h_0<u(t,\cdot)<1/{h_0}$ follows from the \emph{parabolic Maximum Principle}. Moreover, $\LL$ is non-degenerate with smooth coefficients, so that the solution is smooth for any $t>0$.

Next we apply the Maximum Principle to $u'$ as follows. By differentiating~\eqref{LL}, we obtain
\[
\(\LL f\)'=\widetilde{\mathcal L}_{\varepsilon,n}(f')-\,\ell_{\varepsilon,n}'\,f'\quad\mbox{where}\quad\widetilde{\mathcal L}_{\varepsilon,n}\,g:=(1-z^2)\,g''-\,(2\,z+\ell_{\varepsilon,n})\,g'
\]
and observe that
\be{ellprime}
\ell_{\varepsilon,n}'(z)=n-\varepsilon\,(n-d)\,\frac{1+\epsilon+z^2}{(1+\varepsilon-z^2)^2}
\ee
is nonnegative because
\[
\ell_{\varepsilon,n}'(z)\ge n-(n-d)\,\frac{1+\epsilon+z^2}{1+\varepsilon-z^2}\ge\frac{2\,(d-n)}{1+\varepsilon-z^2}\ge0\,.
\]
{}For some $\alpha$ to be chosen later, with $v^\alpha=u^{\beta p}$ and $m$ given by~\eqref{Id:mbeta}, Eq.~\eqref{Eqn:UltraFDreg} can be rewritten as
\[
\alpha\,v^{\alpha-1}\,\frac{\partial v}{\partial t}=\frac1m\,\LL(v^{m\alpha})=\alpha\,v^{m\alpha-1}\(\LL v-\,(1-m\,\alpha)\,\rho^2\,\frac{|v'|^2}v\),
\]
that is,
\[
\frac{\partial v}{\partial t}=v^{(m-1)\,\alpha}\(\LL v-\,(1-m\,\alpha)\,\rho^2\,\frac{|v'|^2}v\).
\]
Let $w:=v'$. By differentiating both sides of the equation with respect to $z$, we obtain
\begin{multline*}
\frac{\partial w}{\partial t}=v^{(m-1)\,\alpha}\(\widetilde{\mathcal L}_{\varepsilon,n}\,w-\,\ell_{\varepsilon,n}'\,w-\,2\,(1-m\,\alpha)\,\rho^2\,\frac{v'}v\,v''\right.\\
\left.+\,(1-m\,\alpha)\,\rho^2\,\frac{|v'|^2\,v'}{v^2}+\,2\,(1-m\,\alpha)\,z\,\frac{|v'|^2}v\)\\
-\,(1-m)\,\alpha\,v^{(m-1)\alpha-1}\,w\,\(\LL v-\,(1-m\,\alpha)\,\rho^2\,\frac{|v'|^2}v\).
\end{multline*}
We can again consider $v$ and $w$ as smooth functions on~$\S^d$ which depend only on $z$, and we know that $w(\pm1)=0$, otherwise $v$ would not be smooth. The function $w$ has a minimum and a maximum.

We first look for an upper bound of $w$. If the maximum is nonpositive, there is nothing to prove as we readily know that $\max w\le0$. Otherwise, for any $t>0$, at a maximum point of~$w$, such that $w(t,\cdot)>0$, then we know that $v''=w'=0$, $\widetilde{\mathcal L}_{\varepsilon,n}\,w=(1-z^2)\,w''\le0$ and $w\,\LL v=-\,\ell_{\varepsilon,n}\,w^2$, so that $\partial w/\partial t$ has the sign~of
\[
(1-z^2)\,w''+\mathsf p_\varepsilon\(\frac wv,z\)\,w
\]
where
\begin{multline*}
\mathsf p_\varepsilon(f,z):=(1-m\,\alpha)\,\big(1+(1-m)\,\alpha\big)\,(1-z^2)\,f^2+\big(2\,(1-m\,\alpha)\,z\\+(1-m)\,\alpha\,\ell_{\varepsilon,n}(z)\big)\,f-\,\ell_{\varepsilon,n}'(z)\,.
\end{multline*}
We learn from~\eqref{Id:mbeta} and~\eqref{RangeBeta} that $2-n\,(1-m)>0$ and $m\neq\frac n{n+2}$. With the choice
\[
\alpha=\frac2{(n+2)\,m-n}\,,
\]
we have $\mathsf p_\varepsilon(f,z)=\mathsf a(z)\,f^2+\mathsf b_\varepsilon(z)\,f+\mathsf c_\varepsilon(z)$ with
\begin{align*}
\mathsf a(z):=&\,(1-m\,\alpha)\,\big(1+(1-m)\,\alpha\big)\,(1-z^2)\\=&\,-\frac{n\,(1-m)\,\big(2-n\,(1-m)\big)}{\big((n+2)\,m-n\big)^2}\,(1-z^2)\,,\\
\mathsf b_\varepsilon(z):=&\,2\,(1-m\,\alpha)\,z+\,(1-m)\,\alpha\,\ell_{\varepsilon,n}(z)=\,\frac{2\,(1-m)}{(n+2)\,m-n}\,\frac{(d-n)\,\varepsilon}{1+\varepsilon-z^2}\,z\,,\\
\mathsf c_\varepsilon(z):=&\,-\,\ell_{\varepsilon,n}'(z)=-\,n+\varepsilon\,(n-d)\,\frac{1+\epsilon+z^2}{(1+\varepsilon-z^2)^2}\,.
\end{align*}
Let us consider the discriminant $\mathsf\delta_\varepsilon(z):=\mathsf b_\varepsilon(z)^2-4\,\mathsf a(z)\,\mathsf c_\varepsilon(z)$. Since
\[
\lim_{\varepsilon\to0_+}\delta_\varepsilon(z)=-\,4\,n^2\,(1-m)\,\frac{2-n\,(1-m)}{((n+2)\,m-n)^2}\,(1-z^2)\,,
\]
then, away from a fixed, small neighbourhood of $z=\pm1$, we know that $\mathsf p_\varepsilon(f,z)<0$ whatever~$f$ is. We recall that $n<d$.

We observe that close to $z=-1$, we also have $\mathsf p_\varepsilon(f,z)<0$ for any $f\ge0$, uniformly with respect to $\varepsilon>0$. In a neighbourhood of $z=+1$, we also have $\mathsf p_\varepsilon(f,z)<0$ if
\[
f(z)<\frac{(n+2)\,m-n}{2\,(1-m)\,z\,\varepsilon}\le-\frac{c_\varepsilon(z)}{b_\varepsilon(z)}\,\big(1+o(1)\big)\quad\mbox{and}\quad
\lim_{\quad\varepsilon\to0_+}\frac{c_\varepsilon(z)}{b_\varepsilon(z)}=-\infty\,.
\]
Altogether, for any given $\bar f$, we can find some $\varepsilon_0=\varepsilon_0(\bar f)>0$ such that
\[
\mathsf p_\varepsilon(f,z)<0\quad\forall\,(f,z,\varepsilon)\in[0,\bar f]\times[-1,1]\times\big(0,\varepsilon_0(\bar f)\big)\,.
\]
The Maximum Principle applies and we learn that the condition
\[
\mbox{sgn}\big(n+2-\beta\,(n+2-p)\big)\,u'(t)\le\frac{|\alpha|}{\beta\,p}\,h_0^{-1}\,\bar f=\frac{h_0^{-1}\,\bar f}{|n+2-\beta\,(n+2-p)|}:=h_1
\]
is stable under the action of the regularized nonlinear flow~\eqref{Eqn:UltraFDreg} as soon as $\varepsilon\in\big(0,\varepsilon_0(\bar f)\big)$. This provides us with an upper bound on $u'(t,\cdot)$.

The lower bound
\[
\mbox{sgn}\big(n+2-\beta\,(n+2-p)\big)\,u'(t)>-\frac{h_0^{-1}\,\bar f}{|n+2-\beta\,(n+2-p)|}=-\,h_1
\]
is obtained by similar estimates, with exactly the same discussion, at a minimum point of~$w$, such that $w(t,\cdot)<0$, which concludes the proof with $\varepsilon_0=\varepsilon_0(\bar f)$ and
\[
\bar f=\big|n+2-\beta\,(n+2-p)\big|\,h_0\,h_1\,.
\]
\end{proof}

\noindent{\bf Remark.} For a generic choice of $\alpha$ in the proof of Lemma~\ref{Lem:Regreg}, we can notice that
\[\textstyle
\lim_{\varepsilon\to0_+}\delta_\varepsilon(z)=4\,n\,(1-m\,\alpha)\,\big(1+(1-m)\,\alpha\big)
\,(1-z^2)+\(2-\big((n+2)\,m-n\big)\,\alpha\)^2z^2
\]
changes sign as a function of $z\in[-1,1)$, unless we choose $\alpha=2\,/((n+2)\,m-n)$. This is the main reason for the choice of $\alpha$ in the proof.

\subsection{Carr\'e du champ method and approximation}\label{Sec:proofUltra}

In the range~\eqref{Range}, we consider the operator defined by~\eqref{LL}. In order to prove Theorem~\ref{Thm:Ultra}, we approximate the ultraspherical operator $\L$ by $\LL$ with $\varepsilon>0$. As a first step, the two identities~\eqref{Gamma2} and~\eqref{L-Gamma} of Lemma~\ref{Lem:TwoId}, which are associated to $\mathcal L_n$, have to be replaced by the corresponding inequalities for $\mathcal L_{\varepsilon,n}$ with $\varepsilon>0$.
\begin{lemma}\label{Lem:TwoIdreg} Assume that $d\ge1$ is an integer, $n\in\R$ and $\varepsilon>0$. For any positive function $u\in C^2$, we have
\be{Gamma2zeta}
\begin{aligned}
\izz{(\LL u)^2}=&\izz{|u''|^2\,\rho^4}+n\izz{\rho^2\,|u'|^2}\\
&-\,\varepsilon\,(n-d)\izz{\frac{1+\varepsilon+\,z^2}{(1+\varepsilon-z^2)^2}\,\rho^2\,|u'|^2}
\end{aligned}
\ee
and
\begin{multline}\label{L-Gammazeta}
\izz{\kern-2pt(\LL u)\,\frac{|u'|^2}u\,\rho^2}=\frac{n}{n+2}\izz{\kern-2pt\frac{|u'|^4}{u^2}\,\rho^4}-2\,\,\frac{n-1}{n+2}\izz{\kern-2ptu''\,\frac{|u'|^2}u\,\rho^4}\\
+\,2\,\varepsilon\,\frac{n-d}{n+2}\izz{\kern-2pt\frac{{u'}^3}u\,\frac{\rho^2\,z}{1+\varepsilon-z^2}}\,.
\end{multline}
\end{lemma}
\begin{proof} For any $\varepsilon>0$, we recall that the operator $\LL$ can be seen as a uniformly elliptic operator with smooth coefficients, acting on a core of symmetric $C^2$ functions defined on~$\S^d$. Here by symmetric functions, we simply mean functions defined on~$\S^d$ which depend only on the $z$-coordinate. Integrations by parts can therefore be performed without any special precautions and no boundary terms have to be taken into account. Taking~\eqref{ellprime} into account, the commutator
\[
\left[\tfrac d{dz},\LL\right]u=-\,2\,z\,u''-\,\ell_{\varepsilon,n}'(z)\,u'=-\,2\,z\,u''-\,n\,u'+\,\varepsilon\,(n-d)\,\frac{1+\varepsilon+\,z^2}{(1+\varepsilon-z^2)^2}\,u'
\]
can be used to prove that
\begin{align*}
&\izz{(\LL u)^2}=-\izz{u'\,(\LL u)'\,\rho^2}\\
&=-\izz{u'\,(\LL u')\,\rho^2}-\izz{u'\(\left[\tfrac d{dz},\LL\right]u\)\rho^2}\\
&=\izz{u''\,\big(u'\,\rho^2\big)'\,\rho^2}\\
&\hspace*{18pt}+\izz{u'\(2\,z\,u''+\,n\,u'-\varepsilon\,(n-d)\,\frac{1+\varepsilon+\,z^2}{(1+\varepsilon-z^2)^2}\,u'\)\rho^2}\\
&=\izz{|u''|^2\,\rho^4}+n\!\izz{\rho^2\,|u'|^2}-\varepsilon\,(n-d)\!\izz{\frac{1+\varepsilon+\,z^2}{(1+\varepsilon-z^2)^2}\,\rho^2\,|u'|^2}\,.
\end{align*}
This completes the proof of~\eqref{Gamma2zeta}.

As for the second identity, we write that
\begin{align*}
&\izz{(\LL u)\,\frac{|u'|^2}u\,\rho^2}=-\izz{u'\(\frac{|u'|^2}u\,\rho^2\)'\,\rho^2}\\
&=\izz{\frac{|u'|^4}{u^2}\,\rho^4}-\,2\izz{u''\,\frac{|u'|^2}u\,\rho^4}+2\izz{\frac{{u'}^3}u\,\rho^2\,z}\,.
\end{align*}
Since
\begin{align*}
&2\izz{\frac{{u'}^3}u\,\rho^2\,z}=-\,\frac2{d+2}\int_{-1}^{+1}\(\rho^{d+2}\)'\,\frac{{u'}^3}u\,\zeta_\varepsilon\,\frac{dz}{Z_{\varepsilon,n}\,Z_d}\\
&=\frac2{d+2}\izz{\(\frac{{u'}^3}u\,\zeta_\varepsilon\)'\,\frac{\rho^4}{\zeta_\varepsilon}}\\
&=\frac6{d+2}\izz{u''\,\frac{|u'|^2}u\,\rho^4}-\,\frac2{d+2}\izz{\frac{|u'|^4}{u^2}\,\rho^4}+\,\frac2{d+2}\izz{\frac{{u'}^3}u\,\frac{\zeta_\varepsilon'}{\zeta_\varepsilon}\,\rho^4}
\end{align*}
and
\[
\frac{\zeta_\varepsilon'}{\zeta_\varepsilon}=-\,(n-d)\,\frac z{\rho^2}+\,\varepsilon\,(n-d)\,\frac z{1+\varepsilon-z^2}\,\frac 1{\rho^2}\,,
\]
we obtain
\begin{multline*}
2\izz{\frac{{u'}^3}u\,\rho^2\,z}=\frac6{n+2}\izz{u''\,\frac{|u'|^2}u\,\rho^4}-\,\frac2{n+2}\izz{\frac{|u'|^4}{u^2}\,\rho^4}\\
+\,2\,\varepsilon\,\frac{n-d}{n+2}\izz{\frac{{u'}^3}u\,\frac{\rho^2\,z}{1+\varepsilon-z^2}}
\end{multline*}
and~\eqref{L-Gammazeta} follows.
\end{proof}

Now let us draw some consequences of Lemma~\ref{Lem:TwoIdreg}. The time derivative of the functional
\[
\mathcal F[u]:=\izz{\rho^2\,\big|(u^\beta)'\big|^2}+\frac\lambda{p-2}\(\big\|u^\beta\big\|_{2,\varepsilon,n}^2-\big\|u^\beta\big\|_{p,\varepsilon,n}^2\)
\]
along the nonlinear the flow~\eqref{Eqn:UltraFDreg} is given with $\gamma:=(\kappa+\beta-1)/(n+2)$ by
\begin{align*}
&\frac1{2\,\beta^2}\,\frac d{dt}\mathcal F[u(t,\cdot)]\\
&=
\lambda\izz{\rho^2\,|u'|^2}-\izz{\(\L u+(\beta-1)\,\frac{|u'|^2}u \rho^2\)\(\L u+\kappa\,\frac{|u'|^2}u \rho^2\)}\\
&=\,(\lambda-n)\izz{\rho^2\,|u'|^2}-\izz{|u''|^2\,\rho^4}\\
&\hspace*{8pt}+\,2\,\frac{n-1}{n+2}\,(\kappa+\beta-1)\izz{u''\,\frac{|u'|^2}u\,\rho^4}\\
&\hspace*{8pt}-\(\kappa\,(\beta-1)+\,\frac n{n+2}\,(\kappa+\beta-1)\)\izz{\frac{|u'|^4}{u^2}\,\rho^4}\\
&\hspace*{8pt}-\,\varepsilon\,(n-d)\(2\,\gamma\izz{\kern-2pt\frac{{u'}^3}u\,\frac{\rho^2\,z}{1+\varepsilon-z^2}}-\izz{\kern-2pt|u'|^2\,\frac{1+\varepsilon+\,z^2}{(1+\varepsilon-z^2)^2}\,\rho^2}\).
\end{align*}
We recall that $\kappa+\beta-1=\beta\,(p-1)$. If $p<2^*=2\,n/(n-2)$, or $p=2^*$ and $n>2$, and $\lambda=n$, then the right-hand side is nonpositive when $\varepsilon=0$ for some choice of $\beta$, according to Lemma~\ref{Lem:Monotonicity-m}. See~\cite{DEKL2014},~\cite[Section~3]{1504} and Appendix~\ref{Appendix} for further details. For $\varepsilon>0$, small enough, we can also control the sign of the right-hand side by taking a slightly smaller value of $\lambda$. The detailed statement goes as follows.
\begin{lemma}\label{Lem:Estim} Assume that $0<n<d$ and take $\beta\neq 0$ be such that~\eqref{deltaneg} holds. Let $u$ be a solution of~\eqref{Eqn:UltraFDreg} with initial datum $u_0$ such that the assumptions of Lemma~\ref{Lem:Regreg} hold for some $h_0\in(0,1)$ and $h_1>0$. With the choice
\[
\lambda=n+\varepsilon\,(n-d)\(\frac{\beta\,(p-1)}{n+2}\,\frac{h_1}{h_0}\)^2\,,\quad\varepsilon\in(0,\varepsilon_0)\,,
\]
where $\varepsilon_0>0$ is obtained in Lemma~\ref{Lem:Regreg}, we have $\frac d{dt}\mathcal F[u(t,\cdot)]\le0$ for any $t\ge0$.
\end{lemma}
\begin{proof} Let $a$ be a positive constant given by
\[
a=\gamma\,\frac{h_1}{h_0}\quad\mbox{with}\quad\gamma:=\frac{\kappa+\beta-1}{n+2}\,,
\]
With the estimate
\begin{multline*}
\left|\izz{\frac{{u'}^3}u\,\frac{\rho^2\,z}{1+\varepsilon-z^2}}\right|\\
\le\frac{h_1}{2\,h_0}\(a\izz{\rho^2\,|u'|^2}+\frac 1a\izz{|u'|^2\,\frac{\rho^2}{(1+\varepsilon-z^2)^2}}\),
\end{multline*}
we obtain that
\begin{multline*}
-\,\varepsilon\,(n-d)\(2\,\gamma\izz{\frac{{u'}^3}u\,\frac{\rho^2\,z}{1+\varepsilon-z^2}}-\izz{|u'|^2\,\frac{1+\varepsilon+\,z^2}{(1+\varepsilon-z^2)^2}\,\rho^2}\)\\
\le\varepsilon\,|n-d|\(\gamma\,\frac{h_1}{h_0}\)^2\izz{\rho^2\,|u'|^2}\,.
\end{multline*}
Taking into account the fact that $\kappa+\beta-1=\beta\,(p-1)$ completes the proof.
\end{proof}

The condition $n<d$ arises from the proof of Lemma~\ref{Lem:Estim}, because the term
\[
\izz{|u'|^2\,\frac{1+\varepsilon+\,z^2}{(1+\varepsilon-z^2)^2}\,\rho^2}
\]
as to carry a negative sign. In a neighbourhood of $z=\pm1$, the weight $\frac{1+\varepsilon+\,z^2}{(1+\varepsilon-z^2)^2}$ is indeed of the order of~$\varepsilon^{-2}$ and we are not aware of an estimate in order to control~it. Hence the case $n>d$ is so far open.

\begin{proof}[Proof of Theorem~\ref{Thm:Ultra} if $p>2^{\#}$.] We choose $\beta$ in the range $\mathcal R(n,p)$ such that~\eqref{RangeBeta} holds and fix an arbitrary $\eta\in(0,n)$. Let us consider a positive function $u_0$ on $[-1,1]$ and assume that, as a symmetric function on~$\S^d$, it is smooth. Let
\[
h_1:=\nrm{u_0'}\infty\quad\mbox{and}\quad h_0=\min\left\{\nrm{u_0}\infty,\,\nrm{u_0^{-1}}\infty^{-1}\right\}\,.
\]
Let us choose $\varepsilon>0$ small enough so that $\lambda$ as in Lemma~\ref{Lem:Estim} satisfies $\lambda\ge n-\eta$. From the monotonicity of $\mathcal F$, we obtain that $\mathcal F[u_0]\ge0$, which means that
\[
\izz{\rho^2\,\big|(u_0^\beta)'\big|^2}\ge\frac{n-\eta}{p-2}\(\big\|u_0^\beta\big\|_{p,\varepsilon}^2-\big\|u_0^\beta\big\|_{2,\varepsilon}^2\).
\]
The inequality being true for any $\eta>0$, it also holds with $\eta=0$. By density of smooth positive functions in the set of nonnegative functions in $\mathrm H^1([-1,1],d\nu_{\kern -0.5pt n})$, we deduce that the inequality holds for any nonnegative function in $\mathrm H^1([-1,1],d\nu_{\kern -0.5pt n})$. Notice that $\beta=\frac n{n-p}$ and $\beta=\frac{n+2}{n+2-p}$ can be obtained as limit cases.
\end{proof}

\appendix\section{On the parameters in the \emph{carr\'e du champ} method}\label{Appendix}

Here we collect elementary computations which determine the range of the parameters for which the Bakry-\'Emery computation gives a sign to the remainder terms. This amounts to find the range of the parameters for which a quadratic form is nonnegative.

\medskip Set $n>0$ and $\beta\neq 0$. Let us recall that $\kappa=\beta\,(p-2)+1$ and to any $C^2$ function~$u$, we associate the function
\[
\mathsf q[u]:=|u''|^2-2\,b\,u''\,\frac{|u'|^2}u+c\,\frac{|u'|^4}{u^2}
\]
with
\[
a=1\,,\quad b=\frac{n-1}{n+2}\,(\kappa+\beta-1)\quad\mbox{and}\quad c=\kappa\,(\beta-1)+\,\frac n{n+2}\,(\kappa+\beta-1)\,.
\]
Then $\mathsf q=\mathsf q[u]$ is pointwise nonnegative under the reduced discriminant condition
\[
\delta:=b^2-a\,c\le0\,.
\]

\smallskip\noindent$\bullet$ If $\beta=1$, $\mathsf q$ is negative if
\[
0\ge\(\frac{n-1}{n+2}\,\kappa\)^2-\frac n{n+2}\,\kappa=\left\{
\begin{array}{ll}
\(\frac{n-1}{n+2}\)^2\,(p-1)\(p-2^\#\)\quad&\mbox{if}\quad n\neq1\,,\\[8pt]
-\frac n{n+2}\,(p-1)\quad&\mbox{if}\quad n=1\,,
\end{array}\right.
\]
that is, for any $p>1$ if $n\le1$ and, if $n>1$, any $p\in[1,2^{\#}]$.

\smallskip\noindent$\bullet$ If $\beta\neq1$, $\mathsf q$ is negative if there exists a value of $\beta\in\R$ such that $\delta$ as defined in~\eqref{deltaneg} is nonpositive. After replacing $\kappa$ by its value, we can consider $\delta$ as a quadratic polynomial of $\beta$, which can be written as
\be{delta}
\delta(\beta)=\(\frac{n-1}{n+2}\,\beta\,(p-1)\)^2-\left[\frac n{n+2}\,\beta\,(p-1)+\big(1+\beta\,(p-2)\big)\,(\beta-1)\right]\,,
\ee
that is, $\delta(\beta)=A\,\beta^2-2\,B\,\beta+C$ where
\[
A=\(\frac{(n-1)\,(p-1)}{n+2}\)^2+2-p\,,\quad B=\frac{n+3-p}{n+2}\quad\mbox{and}\quad C=1\,.
\]
A very special case corresponding to $(n,p)=(3,6)$, has to be excluded, because in that case one has $A=B=0$. It can be handled either by approximation or by considering a power of $u$: see~\cite{DEKL2014,1504} for more details. In case $n>2$, since $\delta(\beta)$ is quadratic in $\beta$ and $\delta(0)=1$, a necessary and sufficient condition for the existence of a $\beta$ such that $\delta(\beta)\le0$ is that the reduced discriminant is nonnegative, which amounts to
\[
B^2-A\,C=\frac{4\,n}{(n+2)^2}\,(p-1)\(2\,n-p\,(n-2)\)\ge0\,.
\]
If $n>2$, notice that this can be rewritten as
\[
B^2-A\,C=\frac{4\,n\,(n-2)}{(n+2)^2}\,(p-1)\(2^*-p\)\ge0\,.
\]
Taking~\eqref{Id:mbeta} into account, the condition $\delta\le0$ determines the \emph{admissible range of~$m$} and amounts to
\be{deltaneg}
m_-(n,p)\le m\le m_+(n,p)\quad\mbox{and}\quad p\le2^*\quad\mbox{if}\quad n>2\,,
\ee
where (see Fig.~\ref{Fig} and also~\cite[Fig.~2]{Dolbeault_2020b})
\[
m_\pm(n,p):=\frac1{(n+2)\,p}\(n\,p+2\pm\sqrt{n\,(p-1)\,\big(2\,n-(n-2)\,p\big)}\).
\]
When $n>2$, note that the square root above is a well defined real number under the assumption $p>1$ only if $p\le2^*$. With $n\le2$, notice that the admissible range is extended to arbitrarily large values of $p$ and that
\[
\lim_{p\to+\infty}m_\pm(n,p)=\frac{n+2\,\sqrt{n\,(2-n)}}{n+2}\,.
\]
The range of the \emph{admissible values of $\beta$} is given by
\[
\mathcal R(n,p)=\big\{\beta\in\R\,:\,m_-(n,p)\le m\le m_+(n,p)\mbox{ where m is given by~\eqref{Id:mbeta} }\big\}\,.
\]
\begin{proposition}\label{Prop:Admissibility} Let $n>0$ and $p>2$. With the above notations, if $\kappa=\beta\,(p-2)+1$, then $\delta\le0$ if and only if $\beta\in\mathcal R(n,p)$.
\end{proposition}
\begin{figure}[ht]
\begin{center}
\includegraphics[width=4cm]{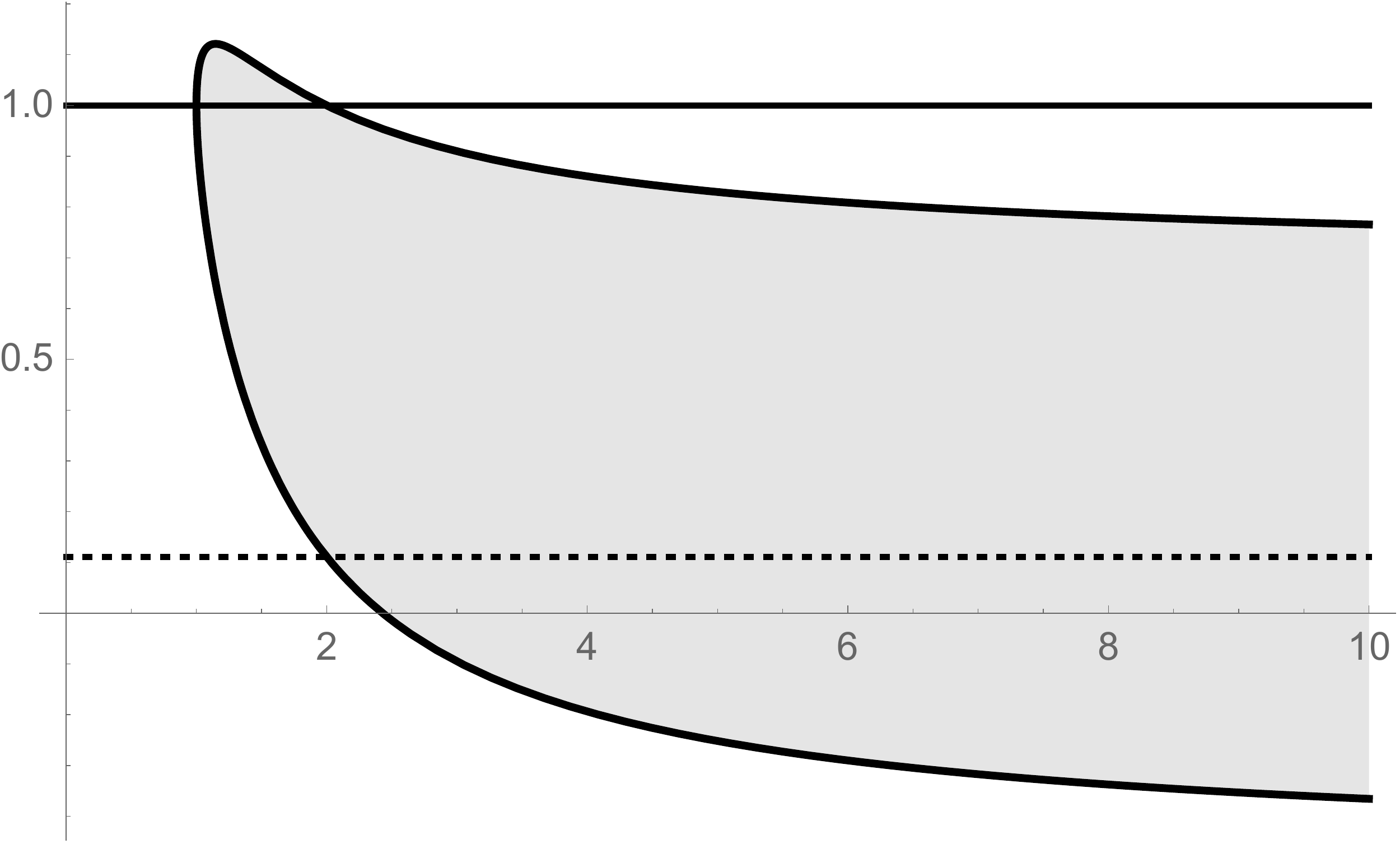}\hspace*{8pt}\includegraphics[width=4cm]{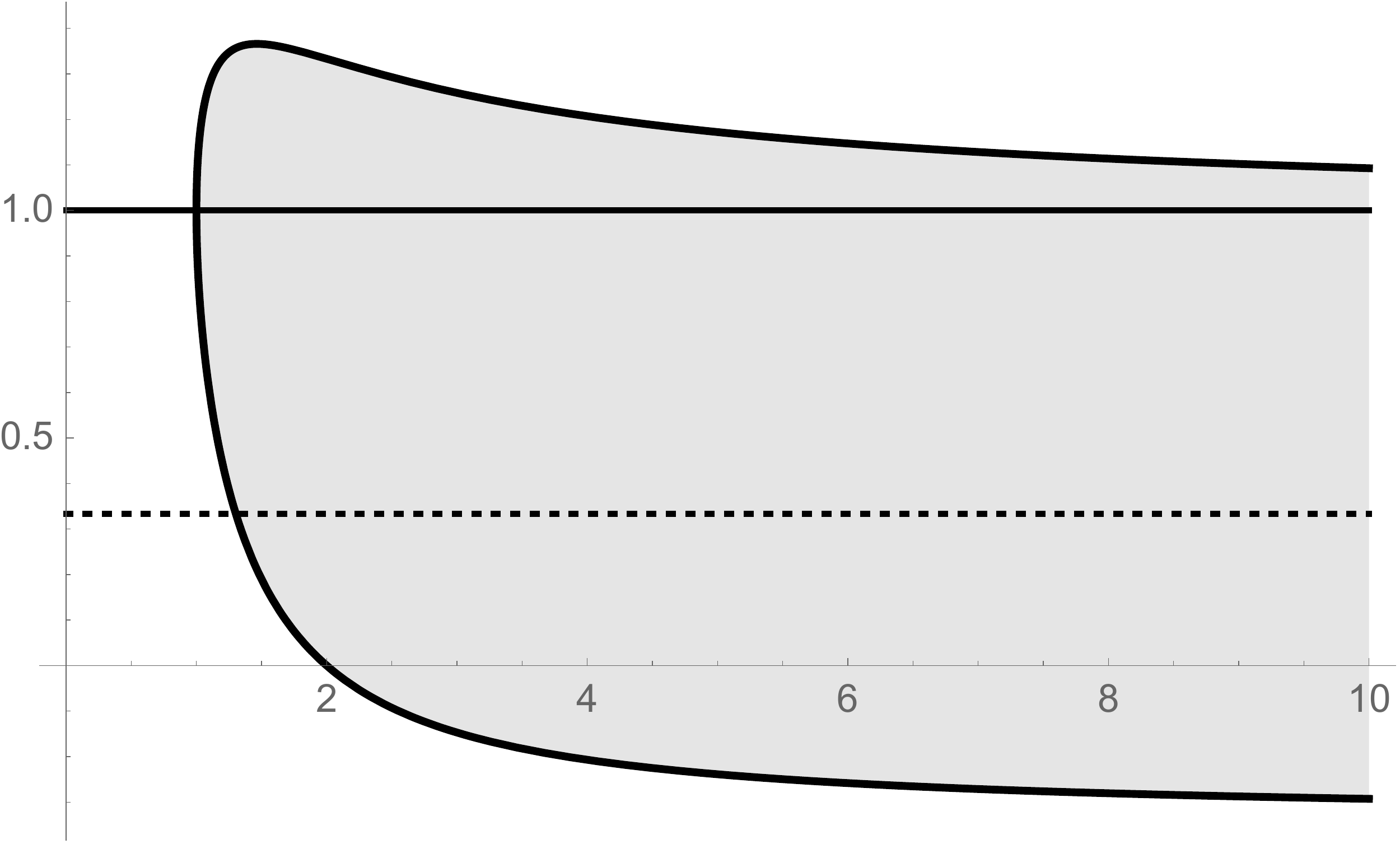}\hspace*{8pt}\includegraphics[width=4cm]{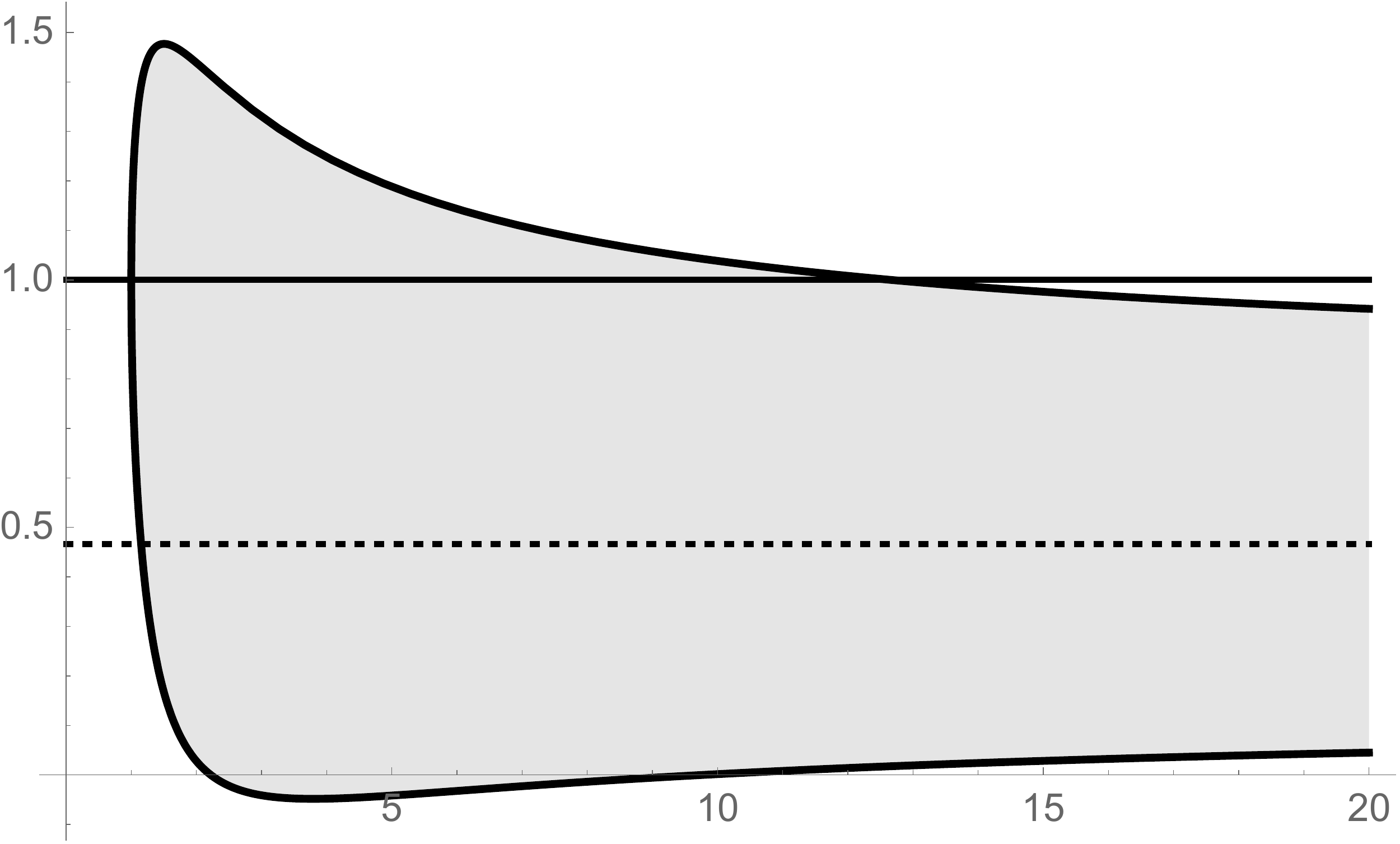}\newline
\includegraphics[width=4cm]{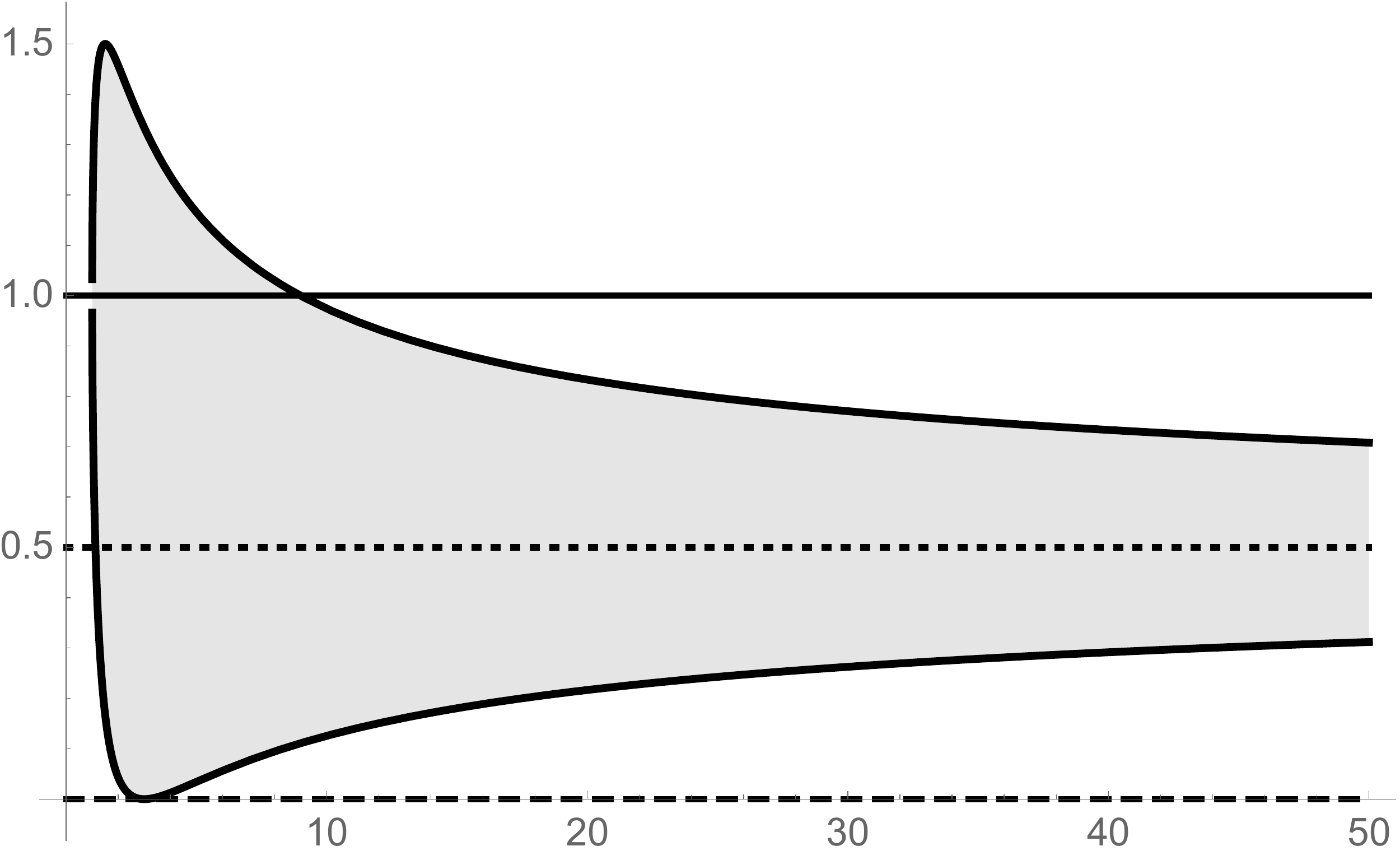}\hspace*{8pt}\includegraphics[width=4cm]{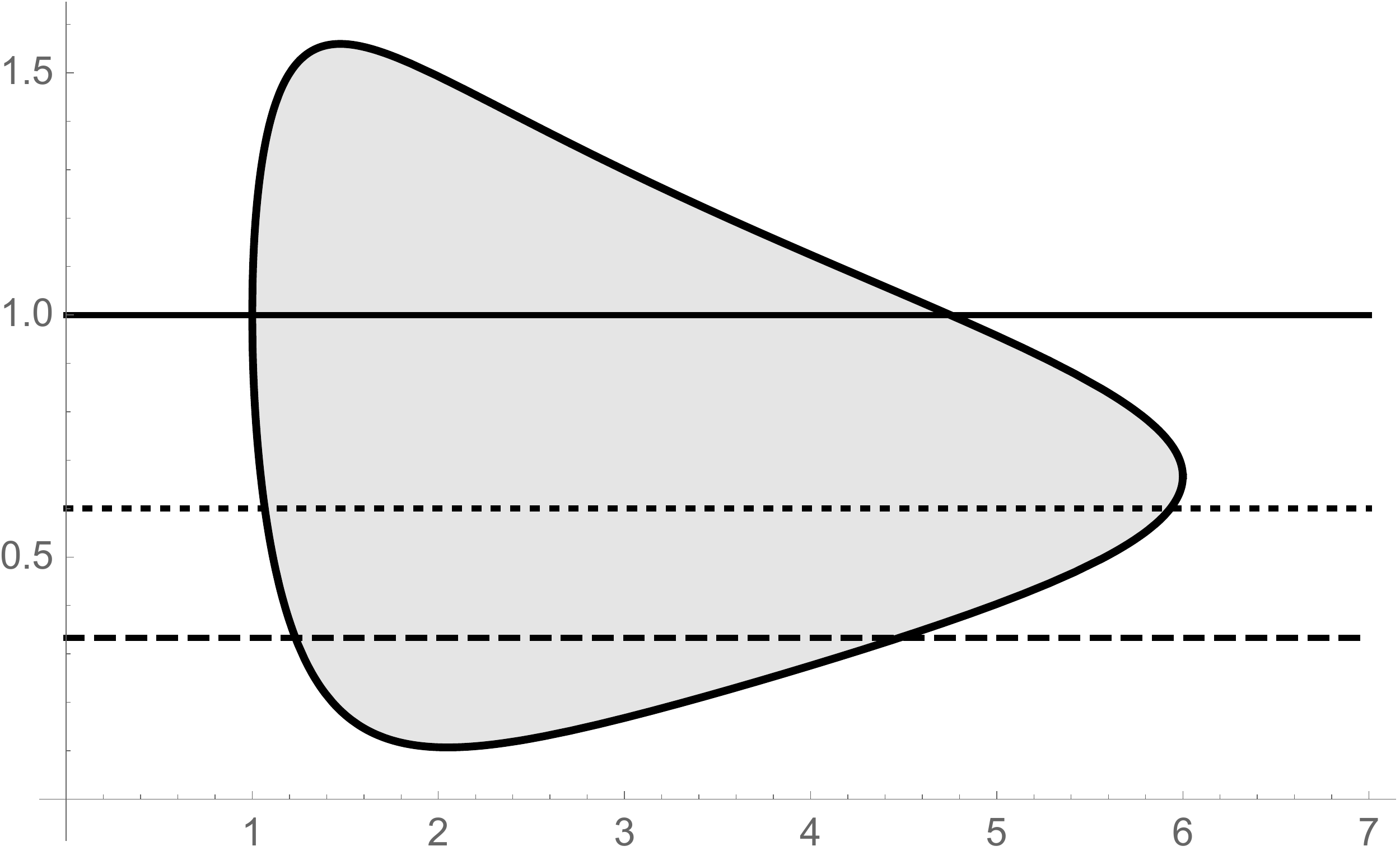}\hspace*{8pt}\includegraphics[width=4cm]{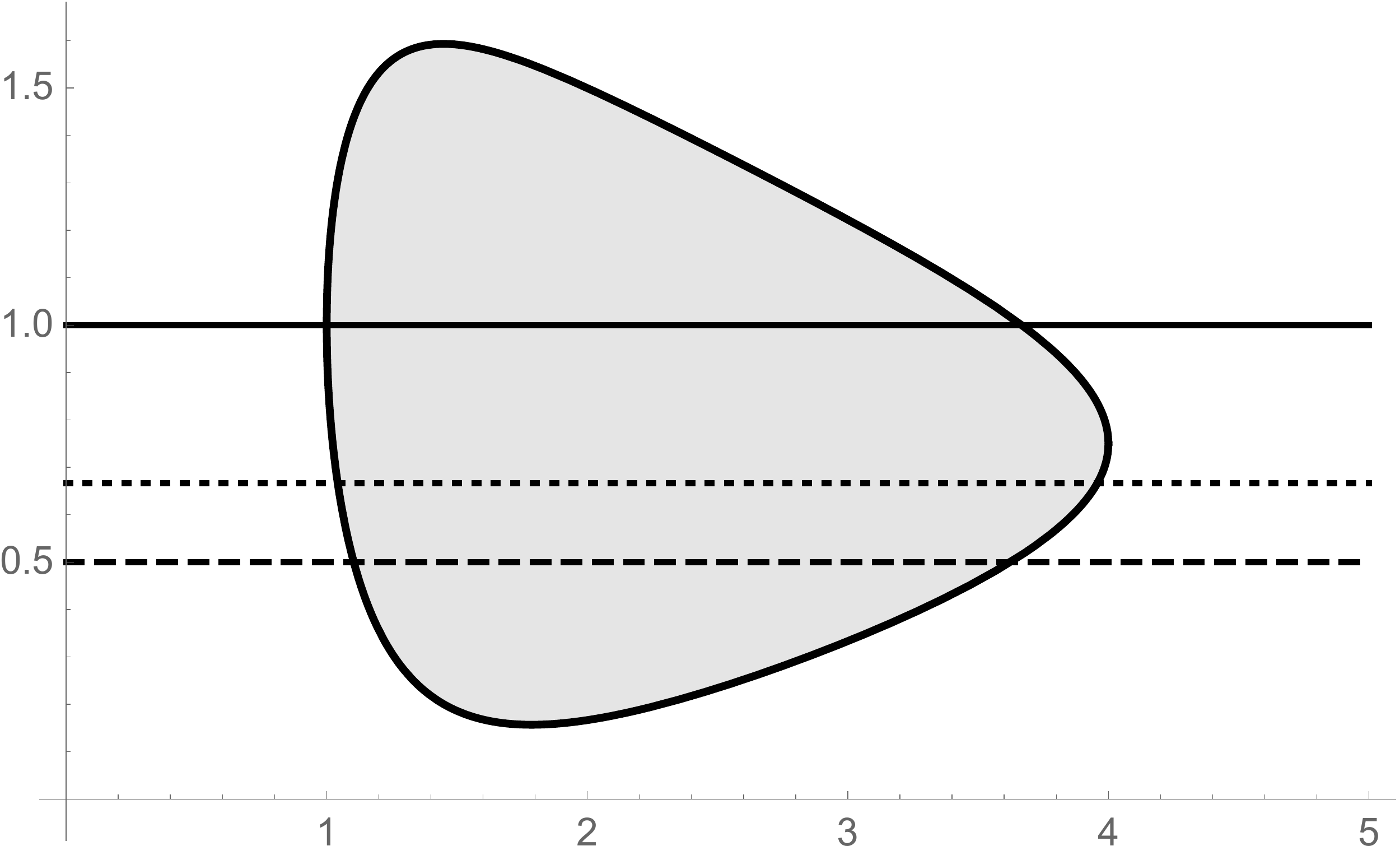}
\end{center}
\caption{\label{Fig} Admissible ranges of $m$ for $n=0.25$, $1$, $1.8$ (first line, from the left to the right) and $n=2$, $3$, $4$ (second line). The dotted and dashed curves correspond respectively to $m=n/(n+2)$ and $m=(n-2)/n$. In Lemma~\ref{Lem:Regreg}, Condition~\eqref{RangeBeta} amounts to $(n-2)/n<m<1$ and $m\neq n/(n+2)$.}
\end{figure}
As a consequence, we have that for any function $u>0$ smooth enough $\mathsf q[u]\ge0$.

\bigskip\begin{spacing}{0.8}{\noindent\scriptsize{\bf Acknowledgments}. This work has been partially supported by the Project EFI (ANR-17-CE40-0030), the NSFC Grant No.~11801536 (A.Z.) and the ERC Advanced Grant BLOWDISOL No.~291214 (A.Z.). The authors thank Nikita Simonov for useful comments on regularity issues.\\ \copyright\,2021 by the authors. This paper may be reproduced, in its entirety, for non-commercial purposes.}\end{spacing}

\providecommand{\href}[2]{#2}
\providecommand{\arxiv}[1]{\href{http://arxiv.org/abs/#1}{arXiv:#1}}
\providecommand{\url}[1]{\texttt{#1}}
\providecommand{\urlprefix}{URL }

\bigskip\begin{center}\rule{2cm}{0.5pt}\end{center}\bigskip
\begin{flushright}\emph{\small\today}\end{flushright}\bigskip
\end{document}